%% file: bootcamp.tex
\documentclass[12pt]{amsart}
\usepackage{amssymb,latexsym,amsmath,comment}
\usepackage[mathscr]{eucal}
\usepackage{bbm,mathdots}
\usepackage{xcolor,tikz}

\numberwithin{equation}{section}
\numberwithin{table}{section}
\numberwithin{figure}{section}

\newtheorem{theorem}[equation]{Theorem}
\newtheorem{corollary}[equation]{Corollary}
\newtheorem{lemma}[equation]{Lemma}
\theoremstyle{remark}
\newtheorem{example}[equation]{Example}
\newtheorem{question}[equation]{Question}
\newtheorem{remark}[equation]{Remark}
\newtheorem{problem}[equation]{Problem}
\newtheorem{exercise}[equation]{Exercise}

\newcommand{\mystack}[2]{
	\ensuremath{ \substack{ \hbox{\tiny{${#1}$}} \\ \hbox{\tiny{${#2}$}} }} }
\def\op{\oplus}
\def\ot{\otimes}

\def\sb{{\hbox{\tiny{$\bullet$}}}}

\def\Kahler{K\"ahler}

\def\a{\alpha}
\def\tand{\quad\hbox{and}\quad}
\def\tAd{\mathrm{Ad}}
\def\tAut{\mathrm{Aut}}
\def\b{\beta}
\def\bC{\mathbb{C}}
\def\d{\delta}
\def\tdeg{\mathrm{deg}}
\def\tdim{\mathrm{dim}}
\def\tEnd{\mathrm{End}}
\def\e{\varepsilon}
\def\cF{\mathcal{F}}
\def\tFlag{\mathrm{Flag}}
\def\tGr{\mathrm{Gr}}
\def\tgr{\mathrm{gr}}
\def\fg{\mathfrak{g}}
\def\bh{\mathbf{h}}
\def\sH{\mathscr{H}}
\def\bi{\mathbf{i}}
\def\tIm{\mathrm{Im}}
\def\tim{\mathrm{im}}
\def\tker{\mathrm{ker}}

\def\cN{\mathcal{N}}
\def\cO{\mathcal{O}}
\def\tO{\mathrm{O}}
\def\bP{\mathbb{P}}

\def\tprim{\mathrm{prim}}
\def\bQ{\mathbb{Q}}
\def\cQ{\mathcal{Q}}
\def\bR{\mathbb{R}}
\def\tRe{\mathrm{Re}}
\def\tSp{\mathrm{Sp}}
\def\tSL{\mathrm{SL}}
\def\s{\sigma}
\def\fs{\mathfrak{s}}
\def\fsl{\mathfrak{sl}}

\def\cV{\mathcal{V}}
\def\w{\omega}

\def\sX{\mathscr{X}}
\def\bZ{\mathbb{Z}}
\def\z{\zeta}

\def\CKpmhs{MR664326}
\def\CKSdeg{MR840721}
\def\Schmid{MR0382272}

\newcounter{icnt}
\newenvironment{i_list}{ 
  \begin{list}{{(\roman{icnt})}}
   {\usecounter{icnt} \setlength{\itemsep}{3pt}
    \setlength{\leftmargin}{25pt} 
    \setlength{\labelwidth}{20pt}
    \setlength{\listparindent}{20pt} }
   }
   {\end{list}}

\newcounter{acnt}

\begin{document}
\title{Degenerations of Hodge structure}
\author[Robles]{C. Robles}
\email{robles@math.duke.edu}
\address{Mathematics Department, Duke University, Box 90320, Durham, NC  27708-0320} 
\thanks{Robles is partially supported by NSF grants DMS 02468621 and 1361120.}
\date{\today}
\begin{abstract}
Two interesting questions in algebraic geometry are: (i) how can a smooth projective varieties degenerate? and (ii) given two such degenerations, when can we say that one is ``more singular/degenerate'' than the other?  Schmid's Nilpotent Orbit Theorem yields Hodge-theoretic analogs of these questions, and the Hodge-theoretic answers in turn provide insight into the motivating algebro-geometric questions, sometimes with applications to the study of moduli.  Recently the Hodge-theoretic questions have been completely answered.  This is an expository survey of that work.
\end{abstract}

\keywords{}
\subjclass[2010]
{
 14D07, 32G20. 
}
\maketitle

\setcounter{tocdepth}{1}


\section{Introduction}

Two motivating questions from algebraic geometry are: 

\begin{question} \label{Q:1}
How can a smooth projective variety degenerate?
\end{question}

\noindent More precisely, let 
\begin{equation}\label{E:f}
  f : \sX \ \to \ S
\end{equation}
be a family of polarized algebraic manifolds.  That is, there is a surjective algebraic mapping $f : \overline\sX \to \overline S$ of complex projective varieties such that the generic fibre $X_s = f^{-1}(s)$ is smooth, $S \subset \overline S$ is the Zariski open subset over which $f$ has smooth fibres and $\sX = f^{-1}(S)$.  Assuming that the family \eqref{E:f} is well-understood, the first question is what can we say about the $X_s$ when $s \in \overline S\backslash S$?  

\begin{question} \label{Q:2}
What are the ``relations" between two such degenerations?  
\end{question}

\noindent The second question is roughly asking for a stratification of $\overline S\backslash S$ with the property that the family is equi-singular along the strata, and ``closure relations" between the strata.  For example, one might consider a case in which the $X_s$ are curves of genus $g$ with at worst nodal singularities, and take the strata to correspond to the number of nodes.

One way to gain insight into Questions \ref{Q:1} and \ref{Q:2} is to ask the analogous questions of invariants associated with the smooth projective varieties.  In this case, the invariant that we have in mind is a polarized Hodge structure.  The Hodge theoretic analogs of the motivating Questions \ref{Q:1} and \ref{Q:2} arise by considering the period map 
\[
  \Phi : S \ \to \ \Gamma \backslash D \,,
\]
which (roughly) assigns to $s \in S$ the Hodge structure on the primitive cohomology on $X_s$.  Here $D$ is a period domain parameterizing Hodge structures, and the assignment is defined up to the action of discrete automorphism group $\Gamma$ of $D$.  (See \S\S\ref{S:hs} and \ref{S:vhs} for details and further references.)  The Hodge theoretic analog of the first motivating question is: how does $\Phi(t)$ degenerate as $t$ approaches a boundary point $s_0 \in \overline S\backslash S$?  Very roughly Schmid's Nilpotent Orbit Theorem \ref{T:norb}, suitably interpreted through results of Cattani, Kaplan and Schmid (Theorem \ref{T:cks}), says that $\Phi(t)$ degenerates to a limit mixed Hodge structure, call it LMHS$(s_0)$.\footnote{The results of Cattani, Kaplan and Schmid hold in the more general setting of abstract variations of Hodge structure (\S\ref{S:vhs}).  Steenbrink \cite{MR0429885} described the limit mixed Hodge structure in the geometric setting, with $\tdim\,\overline S = 1$ and $X_{s_0}$ a normal crossing divisor, in terms of the cohomology of certain intersections of components of $X_{s_0}$.}  (More generally, detailed analysis of degenerations of polarized Hodge structures can be used to better understand degeneration of smooth projective varieties, and moduli spaces and their compactifications, see the surveys \cite{MR2931865, LazaPers} and the references therein.) 

On the other hand, each (not necessarily closed or smooth) algebraic variety carries a mixed Hodge structure by Deligne \cite{MR0498552}.  If DMHS$(s_0)$ denotes Deligne's mixed Hodge structure on $X_{s_0}$, then it is natural to ask how are the two mixed Hodge structures LMHS$(s_0)$ and DMHS$(s_0)$ related?  In the case that $\tdim\,\overline S = 1$ and the family $f:\overline\sX \to \overline S$ is semistable, the two mixed Hodge structures LMHS$(s_0)$ and DMHS$(s_0)$ are related by the the Clemens--Schmid exact sequence \cite{MR0444662}.  

A semisimple Lie group $G \supset \Gamma$ acts homogeneously on $D$, and there is a natural action of $G$ on the set of limit mixed Hodge structures.   (In practice, $G$ is a symplectic $\tSp(2g,\bR)$ or orthogonal $\tO(a,b)$ group.)  This first goal of this survey is to describe a classification of the $G$--conjugacy classes of limit mixed Hodge structures (\S\ref{S:classPMHS}).  This answers the Hodge theoretic analog of Question \ref{Q:1}.

For the Hodge theoretic interpretation of the second question, we consider the case that $\tdim\,\overline S = 2$ and $s_0$ admits a neighborhood $U \subset \overline S$ biholomorphic to a product of unit discs $\Delta \times \Delta$ that identifies $s_0$ with $(0,0)$ and so that $U \cap S = \Delta^* \times \Delta^*$ is a product of punctured discs. 
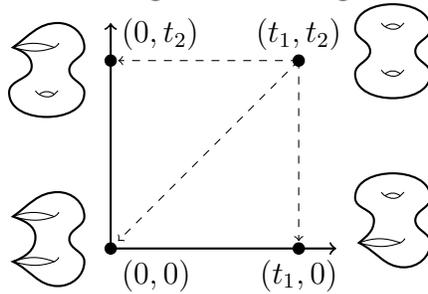
\begin{figure}[!h]
\caption{Degeneration of genus 2 curves}
\begin{tikzpicture}[scale=0.5]
\draw [<->,thick] (0,6) -- (0,0) -- (6,0);
\draw [fill] (5,5) circle [radius=0.15];
\node [above] at (5,5) {$(t_1,t_2)$};
\draw [fill] (5,0) circle [radius=0.15];
\node [below] at (5,0) {$(t_1,0)$};
\draw [fill] (0,5) circle [radius=0.15];
\node [above right] at (0,5) {$(0,t_2)$};
\draw [fill] (0,0) circle [radius=0.15];
\node [below right] at (0,0) {$(0,0)$};
\draw [->,dashed] (5,5) -- (0.2,0.2);
\draw [->,dashed] (5,5) -- (5,0.2);
\draw [->,dashed] (5,5) -- (0.2,5);
\begin{scope}[shift={(7.5,4)}]
  \draw [thick] (0,0) to [out=0,in=270] (1,0.65)
  	to [out=90,in=270] (0.6,1.25)
  	to [out=90,in=270] (1,1.85)
  	to [out=90,in=0] (0,2.5)
  	to [out=180,in=90] (-1,1.85)
  	to [out=270,in=90] (-0.6,1.25)
  	to [out=270,in=90] (-1,0.65)
  	to [out=270,in=180] (0,0);
\begin{scope}[yshift=-0.25cm]
  \draw [thin] (-0.3,1.0) to [out=-60,in=150] (-0.2,0.9) 
	to [out=-30,in=220] (0.2,0.9)
	to [out=40,in=220] (0.3,1.0);
  \draw [thin] (-0.2,0.9) to [out=40,in=140] (0.2,0.9);
\end{scope}
\begin{scope}[yshift=1cm]
  \draw [thin] (-0.3,1.0) to [out=-60,in=150] (-0.2,0.9) 
	to [out=-30,in=220] (0.2,0.9)
	to [out=40,in=220] (0.3,1.0);
  \draw [thin] (-0.2,0.9) to [out=40,in=140] (0.2,0.9);
\end{scope}
\end{scope}
\begin{scope}[shift={(7.5,-0.5)}]
  \draw [thick] (0.3,0) to [out=0,in=270] (1,0.65)
  	to [out=90,in=270] (0.6,1.25)
  	to [out=90,in=270] (1,1.85)
  	to [out=90,in=0] (0,2.5)
  	to [out=180,in=90] (-1,1.85)
  	to [out=270,in=90] (-0.5,1.25)
  	to [out=270,in=40] (-0.9,0.65)
  	to [out=-25,in=180] (0.3,0);
  \draw [thin] (-0.9,0.65) to [out=-15,in=200] (0.2,0.65) 
	to [out=20,in=25] (0.3,0.75);
  \draw [thin] (-0.9,0.65) to [out=20,in=160] (0.2,0.65);
\begin{scope}[yshift=1cm]
  \draw [thin] (-0.3,1.0) to [out=-60,in=150] (-0.2,0.9) 
	to [out=-30,in=220] (0.2,0.9)
	to [out=40,in=220] (0.3,1.0);
  \draw [thin] (-0.2,0.9) to [out=40,in=140] (0.2,0.9);
\end{scope}
\end{scope}
\begin{scope}[shift={(-1.7,3.5)}]
  \draw [thick] (0,0) to [out=0,in=270] (1,0.65)
  	to [out=90,in=270] (0.6,1.25)
  	to [out=90,in=270] (1,1.85)
  	to [out=90,in=0] (0.3,2.5)
  	to [out=180,in=35] (-0.9,1.85)
  	to [out=-20,in=90] (-0.4,1.35)
  	to [out=270,in=90] (-1,0.65)
  	to [out=270,in=180] (0,0);
\begin{scope}[yshift=-0.3cm]
  \draw [thin] (-0.3,1.0) to [out=-60,in=150] (-0.2,0.9) 
	to [out=-30,in=220] (0.2,0.9)
	to [out=40,in=220] (0.3,1.0);
  \draw [thin] (-0.2,0.9) to [out=40,in=140] (0.2,0.9);
\end{scope}
\begin{scope}[yshift=1.2cm] 
  \draw [thin] (-0.9,0.65) to [out=-15,in=200] (0.2,0.65) 
	to [out=20,in=25] (0.3,0.75);
  \draw [thin] (-0.9,0.65) to [out=20,in=160] (0.2,0.65);
\end{scope}
\end{scope}
\begin{scope}[shift={(-1.7,-1)}]
  \draw [thick] (0.3,0) to [out=0,in=270] (1,0.65)
  	to [out=90,in=270] (0.6,1.25)
  	to [out=90,in=270] (1,1.85)
  	to [out=90,in=0] (0.3,2.5)
  	to [out=180,in=35] (-0.9,1.85)
  	to [out=-20,in=90] (-0.3,1.30)
  	to [out=270,in=20] (-0.9,0.65)
  	to [out=-25,in=180] (0.3,0);
  \draw [thin] (-0.9,0.65) to [out=-15,in=200] (0.2,0.65) 
	to [out=20,in=25] (0.3,0.75);
  \draw [thin] (-0.9,0.65) to [out=20,in=160] (0.2,0.65);
\begin{scope}[yshift=1.2cm]
  \draw [thin] (-0.9,0.65) to [out=-15,in=200] (0.2,0.65) 
	to [out=20,in=25] (0.3,0.75);
  \draw [thin] (-0.9,0.65) to [out=20,in=160] (0.2,0.65);
\end{scope}
\end{scope}
\end{tikzpicture}
\label{fig:g2}
\end{figure}
An illustrative example to keep in mind here is the case that $\Delta^* \times \Delta^*$ parameterizes smooth curves of genus 2, with one cycle degenerating to a node as $t_1\to0$, and another cycle degenerating to a node as $t_2 \to0$, c.f.~Figure \ref{fig:g2}.  The idea is that each of the three 1--parameter degenerations $(t_1,t_2) \to (t_1,0)$, $(t_1 , t_2) \to (0,t_2)$ and $(t,t) \to (0,0)$ will give limit mixed Hodge structures LMHS$(t_1,0)$, LMHS$(0,t_2)$ and LMHS$(0,0)$, respectively.  We regard LMHS$(0,0)$ as more degenerate/singular than LMHS$(t_1,0)$ and LMHS$(0,t_2)$ and declare ``polarized relations" LMHS$(t_1,0)$, LMHS$(0,t_2) \prec$ LMHS$(0,0)$.  The second goal of this survey is to classify the polarized relations between (representatives of) $G$--conjugacy classes of limit mixed Hodge structures (\S\ref{S:classpr}).  This answers the Hodge theoretic analog of Question \ref{Q:2}.  

Both classifications we shall discuss are given by discrete combinatorial data in the form of ``Hodge diamonds," weighted configurations of integer points in the $pq$--plane.  

More generally, we can ask these questions in the setting of Mumford--Tate domains.  The latter are generalizations of period domains: they are the classifying spaces of Hodge structures with (possibly) non-generic Hodge tensors \cite{MR2918237}.  As such they are realized as subdomains of period domains.  Unfortunately, once we move to the more general setting of Mumford--Tate domains, the combinatorially simple Hodge diamonds do not suffice to classify the PMHS and the polarized relations amongst them.  The general classifications are given by representation theoretic data (in the form of Weyl groups, Levi subgroups, and embeddings of $\tSL(2)$ into the Mumford--Tate group) that is associated with the domain; see \cite{SL2} and \cite{KPR} for details.  The goal of this article is to give an expository survey of that work in the relative simple setting of period domains.  Related expository articles include \cite{GGR} which studies a coarser notation of polarized relation that is defined in terms of the $G$--orbit structure of the topological boundary of $D$ in the compact dual, and \cite{BPR} which studies the representation theoretic structure of the nilpotent cones underlying a nilpotent orbit.  (As will be discussed in \S\ref{S:no}, nilpotent orbits asymptotically approximate period maps near $s_0 \in \overline S\backslash S$.)

\subsection*{Acknowledgements}
I learnt much of the material presented here from collaboration, correspondence and the work of several colleagues; I am especially indebted to M.~Green, P.~ Griffiths, M.~Kerr, R.~Laza, W.~McGovern and G.~Pearlstein.  

These notes were prepared for the Algebraic Geometry Summer Research Institute Graduate Student Bootcamp in Salt Lake City, Utah, July 06--10, 2015.  I thank the Bootcamp organizers \.{I}.~Co\c{s}kun, T.~de Fernex, A.~Gibney and M.~Leiblich for the opportunity to participate.

\section{Hodge structures and their generalizations} \label{S:PMHS}

We fix, once and for all, a rational vector space $V$, an integer $n$ and a nondegenerate bilinear form $Q : V \times V \to \bQ$ with the property $Q(u,v) = (-1) Q(v,u)$ for all $u,v\in V$.  A brief review of Hodge theory follows; for more see \cite{MR2012297, MR3288678, MR2393625} and the references therein.

\subsection{Hodge structures}  \label{S:hs}

A (pure) \emph{Hodge structure of weight $0\le n \in \bZ$} on the rational vector space $V$ is given by either of the following two equivalent objects:\footnote{It is implicit in this definition that we are assuming that the Hodge structure is effective ($V^{p,q}=0$ if either $p < 0$ or $q < 0$).  Neither this nor the assumption that the weight is non-negative is necessary (or even desirable), but we restrict to this case for notational/expository clarity and convenience.}  
A \emph{Hodge decomposition} 
\begin{equation} \label{E:hd}
  V_\bC \ = \ \bigoplus_{p+q = n} V^{p,q}
  \quad \hbox{such that} \quad \overline{V^{p,q}} = V^{q,p} \,.
\end{equation} 
A (finite, decreasing) \emph{Hodge filtration}
\begin{equation} \label{E:hf}
\begin{array}{c}
  0 \ \subset \ F^{n} \ \subset \ F^{n-1} \ \subset \cdots
  \subset F^1 \ \subset \ F^0 \ = \ V_\bC \\
  \quad\hbox{such that}\quad V_\bC \ = \ F^k \,\op\, \overline{F^{n+1-k}} \,.
\end{array}
\end{equation} 
The equivalence of the two definitions is given by 
\[
  F^k \ = \ \bigoplus_{p\ge k} V^{p,n-p} \tand
  V^{p,q} \ = \ F^p \,\cap\,\overline{F^q} \,.
\]

\begin{example} \label{eg:hs}
The Hodge Theorem asserts that the $n$-th cohomology group $V = H^n(X,\bQ)$ of a compact \Kahler~manifold admits a Hodge structure of weight $n$, with $V^{p,q} = H^{p,q}(X) \subset H^n(X,\bC)$ the cohomology classes in represented by $(p,q)$--forms.
\end{example}

\noindent The \emph{Hodge numbers} $\bh = (h^{p,q})$ and $\mathbf{f} = (f^p)$ are
\[
  h^{p,q} \ := \ \tdim_\bC\,V^{p,q} \tand 
  f^p \ := \ \tdim_\bC\,F^p \,.
\]

A weight $n$ Hodge structure on $V$ is $Q$--\emph{polarized} if the \emph{Hodge--Riemann bilinear relations} hold:
\begin{subequations} \label{SE:HR}
\begin{eqnarray}
  \label{E:hr1}
  Q(V^{p,q},V^{r,s}) \ = \ 0 & \hbox{if} & (p,q) \not= (s,r)\,,\\
  \label{E:hr2}
  \bi^{p-q}Q(v,\bar v) \ > \ 0 & \hbox{for all}& 0\not=v\in V^{p,q} \,.
\end{eqnarray}
\end{subequations}
The \emph{period domain} $D = D_{\bh,Q}$ is the set of all $Q$--polarized Hodge structures on $V$ with Hodge numbers $\bh$.  It is a homogeneous space with respect to the action of the real automorphism group
\[
  G \ := \ \tAut(V_\bR,Q) \,,
\]
and the isotropy group is compact.  If $n$ is odd, then $G \simeq \tSp(2g,\bR)$, where $\tdim\,V = 2g$; if $n = 2k$ is even, then $G \simeq \tO(a,b)$ where $a = \sum h^{k+2p,k-2p}$ and $b = \sum h^{k+1+2p,k-1-2p}$.

\begin{example}
Let $X\subset \bP^m$ be a projective algebraic manifold of dimension $d$ with hyperplane class $\w\in H^2(X,\bZ)$.  Given $n \le d$, the primitive cohomology
\[
  V \ = \ 
  P^n(X,\bQ) \ := \ \{ \a \in H^n(X,\bQ) \ | \ \w^{d-n+1}\wedge\a = 0 \}
\]
inherits the weight $n$ Hodge decomposition $V_\bC = \op_{p+q=n} \, H^{p,q}(X)\cap V_\bC$ from $H^n(X,\bQ)$.   The Hodge--Riemann bilinear relations assert that this Hodge structure is polarized by $Q(\a,\b) := (-1)^{n(n-1)} \int_X \a\wedge\b\wedge\w^{d-n}$.
\end{example}

With respect to the Hodge filtration \eqref{E:hf}, the first Hodge--Riemann bilinear relation \eqref{E:hr1} asserts that $F = (F^p)$ is \emph{$Q$--isotropic}
\begin{equation} \label{E:hr1'}
  Q(F^p,F^q) \ = \ 0 \,,\quad \hbox{for all} \quad p+q = n+1 \,.
\end{equation}
Equivalently, the Hodge filtration defines a point in the rational homogeneous variety 
\[
  \check D \ := \ \tFlag^Q(\mathbf{f},V_\bC)
\] 
of $Q$--isotropic filtrations $F^\sb = (F^p)$ of $V_\bC$; the variety $\check D$ is known as the \emph{compact dual} (\emph{of $D$}).  The complex automorphism group 
\[
  G_\bC \ := \ \tAut(V_\bC,Q)
\]
acts transitively on $\check D$, and contains the period domain $D$ as an open subset.  In summary, the compact dual $\check D$ parameterizes filtrations $F$ of $V_\bC$ satisfying the first Hodge--Riemann bilinear relation, and the period domain $D$ parameterizes filtrations satisfying both Hodge--Riemann bilinear relations.

\subsection{Mixed Hodge structures} 

\subsubsection{Definition and examples}

A \emph{mixed Hodge structure} (MHS) on $V$ is given an increasing filtration $W = (W_\ell)$ of $V$, and a decreasing filtration $F = (F^p)$ of $V_\bC$ with the property that $F$ induces a weight $\ell$ Hodge structure on the graded quotients
\[
  W^\tgr_\ell \ := \ W_\ell/W_{\ell-1} \,.
\]

\begin{example} \label{eg:mhsH}
If $X$ is a \Kahler~manifold of dimension $d$ and 
\[
  V \ = \ H(X,\bQ) \ := \ \bigoplus_n H^n(X,\bQ) \,,
\]
then $W_\ell = \op_{n\le \ell} \, H^n(X,\bQ)$ and $F^k = \op_{p\ge k} \, H^{p,\sb}(X)$ defines a mixed Hodge structure on $V$.
\end{example}

\begin{example} \label{eg:mhsH'}
Alternatively, if $X$ is a \Kahler~manifold of dimension $d$ and $V = H(X,\bQ)$, then $W_\ell = \op_{n\ge 2d-\ell}\,H^n(X,\bQ)$ and $F^k = \op_{q\le d-k} \, H^{\sb,q}(X)$ defines a mixed Hodge structure on $V$.
\end{example}

\begin{example}
Deligne \cite{MR0498552} has shown that the cohomology $H^n(X,\bQ)$ of an algebraic variety $X$ admits a (functorial) mixed Hodge structure.    Here $X$ need not be smooth or closed.  However, when $X$ is smooth and closed, Deligne's MHS is the (usual) Hodge structure of Example \ref{eg:hs}.  For an expository introduction to mixed Hodge structures on algebraic varieties see \cite{MR713069}; for a thorough treatment see \cite{MR2393625}.
\end{example}

\subsubsection{Deligne splitting} \label{S:ds}

Given a mixed Hodge structure $(W,F)$ on $V$ there exists a unique splitting
\begin{subequations}\label{SE:ds}
\begin{equation}
  V_\bC \ = \ \bigoplus I^{p,q}
\end{equation}
with the properties that 
\begin{equation}\label{E:dsFW}
  F^p \ = \ \bigoplus_{p\ge r} I^{r,\sb} \,,\quad
  W_\ell \ = \ \bigoplus_{p+q\le \ell} I^{p,q}
\end{equation}
and 
\begin{equation}
  \overline{I^{p,q}} \ \equiv \ I^{q,p} \quad\hbox{mod}\quad
  \bigoplus_{\mystack{r<q}{s<p}} I^{r,s} \,.
\end{equation}
\end{subequations}
The splitting is given by 
\begin{eqnarray*}
  I^{p,q} & = & F^p \,\cap\, W_{p+q} \,\cap\,
  \left( \overline{F^q}\,\cap\,W_{p+q} \,+\, \overline{U^{q-1}_{p+q-2}}
  \right) \,,\quad\hbox{where} \\
  U^a_b & := & \sum_{j\ge0} F^{a-j} \,\cap\, W_{b-j} \,.
\end{eqnarray*}
Note that \eqref{E:dsFW} implies that the 
\begin{equation} \label{E:dsWl}
  \hbox{Hodge decomposition on $W{}^\tgr_\ell$ induced by $F$ is }
  W^\tgr_\ell \ot \bC \ \simeq \ \bigoplus_{p+q=\ell} \, I^{p,q} \,.
\end{equation}
The MHS is \emph{$\bR$--split} if $\overline{I^{p,q}} = I^{q,p}$.

\begin{example}
The Deligne splittings of the MHS in Examples \ref{eg:mhsH} and \ref{eg:mhsH'} are given by $I^{p,q} = H^{p,q}(X)$ and $I^{p,q} = H^{d-q,d-p}(X)$, respectively.  Both are $\bR$--split.
\end{example}

Let 
\[
  \Lambda^{-1,-1}_\bC \ := \ \left\{ \xi \in \tEnd(V_\bC) \ \left| \ 
  \xi(I^{p,q}) \subset \bigoplus_{\mystack{r<p}{s<q}} I^{r,s} \ \forall \ 
  p,q \right.\right\} \,.
\]
Then $\Lambda^{-1,-1}_\bC$ is a nilpotent subalgebra of $\tEnd(V_\bC)$ and is defined over $\bR$.  Deligne showed that given a MHS $(W,F)$ there exists a unique $\d \in \Lambda^{-1,-1}_\bR$ so that $(W,\tilde F)$, with $\tilde F = e^{\bi\d}F$, is an $\bR$--split PMHS.  An important property of this new $\tilde F$ is that it determines the same Hodge structure on $W^\tgr_\ell$ as the original $F$.  In particular, if $V_\bC = \op \tilde I^{p,q}$ is the Deligne splitting for $(W,\tilde F)$, then 
\begin{equation}\label{E:dim}
  \tdim_\bC\,I^{p,q} \ = \ \tdim_\bC\,\tilde I^{p,q} \,,
\end{equation}
for all $p,q$.

\subsection{Polarized mixed Hodge structures} 

\subsubsection{Jacobson--Morosov filtrations} \label{S:jmf}

Every nilpotent endomorphism $N : V \to V$ determines a unique increasing filtration $W(N) = (W_\ell(N))$ of $V$ with the two properties that
\begin{subequations} \label{SE:W(N)}
\begin{equation}\label{E:N(W)}
  N (W_\ell(N)) \ \subset \ W_{\ell-2}(N)
\end{equation}
and 
\begin{equation}\label{E:Wgr-isom}
  \hbox{the induced $N^\ell : W^\tgr_\ell(N) \to W^\tgr_{-\ell}(N)$ 
  is an isomorphism for all $\ell \ge 0$.}
\end{equation}
\end{subequations}
Moreover, if $N$ lies in the Lie algebra
\[
  \fg \ := \ \tEnd(V,Q)
\]
of $G = \tAut(V,Q)$, then the filtration $W(N)$ is $Q$--isotropic.

\begin{exercise}
Suppose that $N^k \not=0$ and $N^{k+1} = 0$.  Show that $W(N)$ is given inductively by 
\begin{eqnarray*}
  W_k \ = \ \tker\,N^{k+1} \ = \ V
  & \hbox{and} & W_{-k-1} \ = \ \tim\,N^{k+1} \ = \ 0 \,,\\
  W_{k-1} \ = \ \tker\,N^k
  & \hbox{and} &
  W_{-k} \ = \ \tim\,N^k \,,
\end{eqnarray*}
and for all $0 \le \ell \le k-2$, 
\[
  W_\ell \ = \ \{ v \in W_{\ell+1} \ | \ N^\ell v \subset W_{-\ell-2} \} 
  \tand
  W_{-\ell-1} \ = \ N^{\ell+1}(W_\ell) \,.
\]
\end{exercise}

Notice that the first nontrivial subspace $W_{-k}(N)$ in the Jacobson--Morosov filtration is indexed by a negative integer (if $N \not=0$).  The ``shifted'' filtration, with nontrivial subspaces indexed by nonnegative integers is denoted $W(N)[-k]$, and given by 
\[
  W_\ell(N)[-k] \ := \ W_{\ell-k}(N) \,.
\]

\begin{remark} \label{R:sl2}
It is sometimes useful to describe the Jacobson--Morosov filtration in terms of the action of a three-dimensional subalgebra $\fs \subset \tEnd(V)$ that is isomorphic to $\fsl(2)$ and contains $N$.  Specifically, the Jacobson--Morosov Theorem asserts that there exist $Y , N^+ \in \tEnd(V)$ so that 
\begin{equation}\label{E:sl2}
  [Y,N] \ = \ -2N \,,\quad [N^+,N] \ = \ Y \tand
  [Y,N^+] \ = \ 2N^+ \,.
\end{equation}
When $N \in \fg$, we can choose $Y,N^+ \in \fg$.  The relations \eqref{E:sl2} imply that $\{ N,Y,N^+\}$ span a subalgebra of $\tEnd(V)$ that is isomorphic to $\fsl(2)$.  Moreover, the element $Y$ acts on $V$ by integer eigenvalues.  If
\[
  V \ = \ \bigoplus_{\ell} V_\ell
\]
is the $Y$--eigenspace decomposition of $V$, then
\[
  W_\ell(N) \ = \ \bigoplus_{m\le\ell} V_m \,.
\]
The Jacobi identity and \eqref{E:sl2} imply $N(V_\ell) \subset N_{\ell-2}$; from this we see that \eqref{E:N(W)} holds.  Note that $W^\tgr_\ell(N) \simeq V_\ell$.  It is a classical result from the representation theory of $\fsl(2)$ that $N^\ell : V_\ell \to V_{-\ell}$ is an isomorphism for all $\ell \ge0$; from this we see that \eqref{E:Wgr-isom} holds.
\end{remark}

\subsubsection{Definition and examples}

A \emph{$Q$--polarized mixed Hodge structure} (PMHS) on $V$ is given by a mixed Hodge structure $(W,F)$ and a set $\cN \subset \fg_\bR$ of nilpotent elements with the properties:
\begin{i_list}
\item For all $N \in \cN$ we have $N^{n+1} = 0$ and $W = W(N)[-n]$.
\item The filtration $F$ is $Q$--isotropic, and $N(F^p) \subset F^{p-1}$ for all $N \in \cN$ and $p$.
\item The filtration $F$ induces a weight $n+\ell$ Hodge structure on the \emph{primitive space}
\[
  P(N)_\ell \ := \ \tker\, \{ N^\ell : W^\tgr_{n+\ell} \to W^\tgr_{n-\ell-2} \}
\]
that is polarized by 
\[
  Q^N_\ell(\cdot,\cdot) \ := \ Q(\cdot , N^\ell \cdot) \,,
\]
for all $\ell \ge0$.
\end{i_list}
We sometimes say that the mixed Hodge structure $(W,F)$ is \emph{polarized by $\cN$}.  From (i) we see that the filtration $W$ is determined by $\cN$, and we will often write $(F,\cN)$ for the PMHS $(W,F,\cN)$.

\begin{example}
Let $X\subset \bP^m$ be a projective algebraic manifold of dimension $n$ with hyperplane class $\w\in H^2(X,\bZ)$.  Let $V = H(X,\bQ)$ and define $Q(\a,\b) = (-1)^{k(k-1)/2} \int_X \a\wedge\b$, with $k = \tdeg\,\a$.  The ray $\cN = \{ t \w \ | \ t > 0 \}$ spanned by the Lefschetz operator $\w : H^\sb(X,\bQ) \to H^{\sb+2}(X,\bQ)$ polarizes the mixed Hodge structure defined of Example \ref{eg:mhsH'}.  (Alternatively, there is a canonical dual $N_\w^* : H^\sb(X,\bQ) \to H^{\sb-2}(X,\bQ)$ to the Lefschetz operator, and the mixed Hodge structure of Example \ref{eg:mhsH} is polarized by the ray $\cN = \{ t \,N_\w^* \ | \ t > 0 \}$.)
\end{example}

\begin{remark}
Given a PMHS $(W,F,\cN)$, let $(W,\tilde F = e^{\bi\d}F)$ be the $\bR$--split MHS of \S\ref{S:ds}.  Then $(W,\tilde F,\cN)$ is a PMHS, and $\tilde F$ determines the same $Q^N_\ell$--polarized Hodge structures on $P(N)_\ell$ as $F$ for all $\ell \ge 0$.  Moreover, $\d \in \Lambda^{-1,-1} \cap \fg_\bR$ and $[\d,N] = 0$ for all $N \in \cN$.  Finally, if $V_\bC = \op I^{p,q}$ is the Deligne splitting for $(W,\tilde F)$, then $N(I^{p,q}) \subset I^{p-1,q-1}$ for all $N \in \cN$.
\end{remark}

\subsection{Variation of Hodge structure} \label{S:vhs}

Let $S$ be a complex manifold with fundamental group $\pi_1(M)$ and universal cover $\tilde S$.  Let $\rho: \pi_1(S) \to \tAut(V,Q)$ be a representation of the fundamental group.  Then 
\[
  \cV \ := \ \tilde S \times_{\pi_1(S)} V
\]
defines a flat vector bundle over $S$.  Let $\nabla$ denote the flat connection.  The bilinear form $Q$ induces a flat form $\cQ$ on $\cV$.  A (\emph{polarized}) \emph{variation of Hodge structure} (VHS) {over $S$} is given by a holomorphic filtration $\cF^\sb$ of $\cV_\bC$ that defines a $\cQ_s$--polarized Hodge structure on each fibre $\cV_s$, $s \in S$, and with the property that $\nabla \cF^p \subset \Omega^1_S \ot \cF^{p-1}$.  The variation of Hodge structure induces a \emph{period map}
\[
  \Phi : S \to \Gamma \backslash D \,,
\]
where $\Gamma = \rho(\pi_1(S)) \subset \tAut(V,Q)$.  Geometrically, VHS arise when considering a family $\sX \to S$ of polarized algebraic manifolds: one obtains a VHS $\cV \to S$ with fibres $\cV_s$ isomorphic to the primitive cohomology $P^n(\sX_s,\bQ)$, and $\cF_s$ the Hodge filtration, see \cite{MR0229641, MR0233825}.

\section{Nilpotent orbits} \label{S:no}

The significance of nilpotent orbits comes from Schmid's Nilpotent Orbit Theorem (\S\ref{S:snot}) which asserts that that every (lifting of a) period map (\S\ref{S:vhs}) is well approximated by a nilpotent orbit.  In particular, the asymptotic behavior of a period mapping is encoded by nilpotent orbits.  Moreover, results of Cattani, Kaplan and Schmid imply that a nilpotent orbit is equivalent to a PMHS (Theorem \ref{T:cks}); this is the sense in which 
\begin{center}
\emph{a PMHS arises from a degeneration of Hodge structure.}
\end{center}

\subsection{Definition}

A \emph{nilpotent orbit} is a map $\theta : \bC^r \to \check D$ of the form
\[
  \theta(z) \ = \ \exp(\sum z_j N_j)\cdot F \,,
\]
with $F \in \check D$ and $\{N_1,\ldots,N_r\}\subset \fg_\bR$ a set of commuting nilpotent elements, and having the properties that:
\begin{i_list}
\item $N_j(F^p) \subset F^{p-1}$ for all $j,p$, and 
\item $\theta(z) \in D$ when $\tIm(z_j) \gg0$ for all $j$.
\end{i_list}

\subsection{Schmid's Nilpotent Orbit Theorem} \label{S:snot}

Fix a VHS $(\cV,\cQ,\cF)$ over $S$ as in \S\ref{S:vhs}, and let $\Phi : S \to \Gamma\backslash D$ denote the associated period map.  In practice one is interested in the case that $S$ is a Zariski open subset of a compact analytic space $\overline S$, and wants to describe the singularities of $\Phi$ on the boundary $\overline S\backslash S$.  Applying Hironaka's resolution of singularities \cite{MR0199184}, we may assume that $\overline S$ is smooth.  If $\overline S \backslash S$ has codimension greater than two, then $\Phi$ extends holomorphically to $\overline S$ \cite{MR0259958}.  In the case that the boundary has codimension one, we may again apply Hironaka's resolution of singularities to assume that $\overline S\backslash S$ is locally a normal crossing divisor.  That is, every point $s \in \overline S$ admits a neighborhood of the form $\Delta^m$ and with the property that $\Delta^m \cap S = (\Delta^*)^r \times \Delta^{m-r}$; here 
\[
  \Delta \ := \ \{ t \in \bC \ : \ |t| < 1 \}
\]
is the unit disc, and 
\[
  \Delta^* \ := \ \{ t \in \bC \ : \ 0 < |t| < 1 \} 
\]
is the punctured unit disc.

The nilpotent orbit theorem is a local statement, describing the behavior of the period map $\Phi(t)$ as $t \to t_o \in \overline S\backslash S$, so we now restrict $\Phi$ to $(\Delta^*)^r \times \Delta^{m-r}$.  For simplicity of exposition we will take $m=r$ and consider the period map
\begin{equation}\label{E:Phi}
  \Phi : (\Delta^*)^r \ \to \ \Gamma\backslash D \,,
\end{equation}
with $\Gamma = \rho(\pi_1( (\Delta^*)^r)) \subset \tAut(V,Q)$; for the general statement of the Nilpotent Orbit Theorem see \cite{MR0382272}.  The fundamental group $\pi_1( (\Delta^*)^r)$ is generated by elements $\{ \gamma_1',\ldots,\gamma_r'\}$ where $\gamma_j'$ may be identified with the counter-clockwise generator of the fundamental group of the $j$--th copy of $\Delta^*$ in $(\Delta^*)^r$.  The images $\gamma_j  = \rho(\gamma_j') \in \Gamma$ are the \emph{monodromy transformations}.  They pairwise commute and are known to be quasi-unipotent.\footnote{In the geometric setting quasi-unipotency is due to Katz \cite{MR0291177} and Landman \cite{MR0344248}; the general statement is due to Borel \cite[(4.5)]{MR0382272}.}  Quasi-unipotency implies there exist $0 \le m_j \in \bZ$ and nilpotent $N_j \in \fg_\bR$ so that 
\[
  \gamma_j^{m_j} \ = \ \exp(m_j\,N_j) \,.
\] 

Let 
\[
  \sH \ := \ \{ z \in \bC \ | \ \tIm\,z > 0 \}
\]
denote the upper-half plane, so that $\sH \to \Delta^*$, sending $z \mapsto t = e^{2\pi\bi z}$, is the universal covering map.  Then $\gamma_j'$ acts on $\sH^r$ by 
\[
  \gamma'_j \cdot (z_1,\ldots,z_r) \ = \ 
  (z_1,\ldots,z_{j-1}, z_j+1 , z_{j+1} , \ldots , z_r ) \ =: \ z + \e_j
\]
by the translation replacing the $j$--th coordinate $z_j$ with $z_j+1$.  Fixing a lift $\tilde\Phi : \sH^r \to D$ of the period map \eqref{E:Phi}, we have $\gamma_j\cdot\tilde\Phi(z) = \tilde\Phi(z+\e_j)$.  In particular, the map
\[
  \widetilde\Psi : \sH^r \ \to \ \check D \quad\hbox{sending}\quad
  z \ \mapsto \ 
  \exp\left(-\textstyle{\sum_j}\, m_j z_j N_j\right) \cdot \tilde\Phi(z)
\]
descends to a well-defined map $\Psi : (\Delta^*)^r \to \check D$.

\begin{theorem}[Schmid's Nilpotent Orbit Theorem {\cite{MR0382272}}] \label{T:norb}
The map $\Psi$ extends holomorphically to $\Delta^r \to \check D$.  Setting $F := \Psi(0)$, the map $\theta(z) := \exp( \sum_j\, z_j N_j) \cdot F$ is a nilpotent orbit.  Moreover, given any $G$--invariant distance $d$ on $D$, there exist constants $0 \le \a,\b,C$ such that $\theta(z) \in D$ and 
\begin{equation}\label{E:bound}
  d \left( \tilde\Phi(z) , \theta(z) \right) \ \le \ 
  C \,\textstyle{\sum_j}\,(\tIm\,z_j)^\b\,
  \exp (-2\pi(\tIm\,z_j)/m_j )
\end{equation}
so long as $\tIm\,z_j > \a$.
\end{theorem}

\noindent The bound \eqref{E:bound}\footnote{This bound is an improvement, due to Deligne, of the initial distance estimate given in \cite{MR0382272}.} is the precise sense in which the nilpotent orbit $\theta$ strongly approximates the lifted period map $\tilde\Phi$ as $\tIm\,z_j \to \infty$.  The constants $\a,\b,C$ depend only on $d$, the $m_j$, the Hodge numbers $\bh = (h^{p,q})$ and the weight $n$.  

\subsection{Relationship to PMHS}

Fix $F \in \check D$ and pairwise commuting nilpotent $N_1 , \ldots , N_r \in \fg_\bR$.  Let 
\[
  \s \ := \ \left\{ \left. \textstyle{\sum_j}\, x_j N_j \ \right| \
  x_j > 0 \right\} \ \subset \ \fg_\bR
\]
be the \emph{nilpotent cone} spanned by the $\{ N_j\}$.

\begin{theorem}[Cattani,Kaplan, Schmid] \label{T:cks}
The map $\theta : \bC^r \to \check D$ sending
\[
  z \ \mapsto \ \exp\left(\textstyle{\sum}_j\, z_j N_j\right)\cdot F 
\]
is a nilpotent orbit if and only if the Jacobson--Morosov filtration $W(N)$ is independent of $N \in \s$, so that $W(\s)$ is well-defined, and $(W(\s)[-n],F,\s)$ is a PMHS.  Moreover, if $(W(N)[-n],\tilde F)$ is the $\bR$--split PMHS constructed by Deligne \emph{(\S\ref{S:ds})}, then the nilpotent orbits $\theta(z)$ and $\tilde\theta(z) = \exp(\sum z_j N_j)\cdot \tilde F$ agree to first order as $\tIm\,z_j \to \infty$.
\end{theorem}

\begin{remark} \label{R:asymlim}
The common asymptotic limit 
\[
  \Phi_\infty(\s,F) \ := \ \lim_{\mystack{\tIm\,z\to\infty}{\tRe\,z \ \mathrm{bdd}}} \exp(z N) \cdot F \,,
\]
with $\tRe\,z$ bounded and $z \in \bC$, of the two nilpotent orbits is independent of our choice of $N \in \s$, \cite{\CKSdeg}; as a filtration of $V_\bC$, the point $\Phi_\infty(\s,F) \in \check D$ is given by the Deligne splitting \eqref{SE:ds} as 
\[
  \Phi_\infty^p(\s,F) \ = \ \bigoplus_{q \le n-p} I^{\sb,q} \,.
\] 
Moreover, note that $\exp(\z N) \cdot F \in D$ for all $\tIm\,\z\gg0$ implies $\Phi_\infty(\s,F)$ lies in the topological closure $\overline D \subset \check D$ of $D$.
\end{remark}

\begin{proof}
Given a one--variable nilpotent orbit $\exp(zN)F$, Schmid \cite{\Schmid} proved that $(F,N)$ is a PMHS.  Given a several--variable nilpotent orbit, the independence of the Jacobson--Morosov filtration $W(N)$ of the choice of $N$ in the underlying nilpotent cone was proven by Cattani and Kaplan \cite{\CKpmhs}.  From these two results it follows that a several--variable nilpotent orbit determines a PMHS.  The converse was proved by Cattani, Kaplan and Schmid \cite{\CKSdeg}.  The asymptotic first-order agreement of $\theta$ and $\tilde\theta$ is also established in \cite{\CKSdeg}.
\end{proof}

\section{Classifications}

\subsection{Classification of $\bR$--split PMHS} \label{S:classPMHS}

Notice that $G$ acts on the set of PMHS: given $g \in G$ and a PMHS $(W,F,\cN)$ we have 
\[
  g \cdot (W,F,\cN) \ := \ ( g \cdot W , g \cdot F , \tAd_g \cN) \,.
\]
Moreover, $(W,F,\cN)$ is $\bR$--split if and only if $g \cdot (W,F,\cN)$ is.  In this section will classify the $\bR$--split PMHS.  The classification is given by Hodge diamonds, which depend only on the MHS $(W,F)$, and it is a consequence of \eqref{E:dim} that $(W,F)$ and $(W,\tilde F)$ have the same Hodge diamond.

Given a MHS $(W,F)$, let $V_\bC = \op I^{p,q}_{W,F}$ be the Deligne splitting (\S\ref{S:ds}).  The \emph{Hodge diamond} of $(W,F)$ is the function $\Diamond(W,F) : \bZ \times \bZ \to \bZ$ given by 
\[
  \Diamond(W,F)(p,q) \ := \ \tdim_\bC \, I^{p,q} \,.
\]

\begin{lemma}[{\cite{KPR}}] \label{L:HD}
The Hodge diamond $\Diamond = \Diamond(W,F,\cN)$ of a PMHS on a period domain $D$ parameterizing weight $n$ Hodge structures with Hodge numbers $\bh = (h^{p,q})_{p+q=n}$ satisfies the following four properties:  The columns of the Hodge diamond sum to the Hodge numbers
\begin{subequations} \label{SE:hd}
  \begin{equation} \label{E:hd1}
    \textstyle{\sum_p} \, \Diamond(p,q) \ = \ h^{n-q,q} \,.
  \end{equation}
The Hodge diamond is symmetric about the diagonal $p=q$:
  \begin{equation} \label{E:hd2}
    \Diamond(p,q) \ = \ \Diamond(q,p) \,.
  \end{equation}
The Hodge diamond is symmetric about $p+q = n$:
  \begin{equation}\label{E:hd3}
    \Diamond(p,q) \ = \ \Diamond(n-q,n-p) \,.
  \end{equation}
The values $\Diamond(p,q)$ are non-increasing as one moves away from $p+q=n$ along a(n off) diagonal:
  \begin{equation} \label{E:hd4}
    \Diamond(p,q) \ \ge \ \Diamond(p+1,q+1) \quad \hbox{for all} 
    \quad p+q \ge n \,.
  \end{equation}
\end{subequations}
\end{lemma}

\noindent Note that the four conditions \eqref{SE:hd} imply that the Hodge diamond of a PMHS ``lies in'' the square $[0,n]\times[0,n]$; that is
\[
  \Diamond(p,q) \ \not= \ 0 \quad\hbox{implies}\quad 0 \,\le\, p,q \,\le\, n \,.
\]

\begin{proof}
The property \eqref{E:hd1} follows from $F \in \check D$ and the first equation of \eqref{E:dsFW}; property \eqref{E:hd2} is due to \eqref{E:dsWl}; and properties \eqref{E:hd3} and \eqref{E:hd4} to \eqref{E:Wgr-isom}.
\end{proof}

\begin{remark}
The Hodge diamond of an \emph{arbitrary} MHS (that is, a MHS that is not necessarily polarized) will satisfy \eqref{E:hd1} and \eqref{E:hd2}, but need not satisfy \eqref{E:hd3} and \eqref{E:hd4}.
\end{remark}

Given a PMHS $(F,\cN)$, we will denote the Hodge diamond by $\Diamond(F,\cN)$.  The following proposition asserts that (i) every non-negative function satisfying \eqref{SE:hd} may be realized as the Hodge diamond of an $\bR$--split PMHS, and (ii) the $\bR$--split PMHS on $D$ are classified, up to the action of $G$, by their Hodge diamonds.

\begin{theorem}[{\cite{KPR}}] \label{T:hd}
Any function $f: \bZ \times \bZ \to \bZ_{\ge0}$ satisfying \eqref{SE:hd} may be realized as the Hodge diamond $\Diamond(F,N)$ of an $\bR$--split polarized mixed Hodge structure $(F,N)$, $N\in\fg_\bR$, on the period domain $D$.  Moreover, $\Diamond(F_1,N_1) = \Diamond(F_2,N_2)$ if and only if $(F_2,N_2) = (g\cdot F_1, \tAd_g N_1)$ for some $g \in G$.
\end{theorem}

\noindent The proof is essentially a consequence of the classification of nilpotent $N \in \fg_\bR$ by ``signed Young diagrams," and the fact that the latter are determined by the Hodge diamonds; see \cite{KPR} for details.

\begin{remark}
By virtue of the equivalence between $\bR$--split PMHS and horizontal $\tSL(2)$--orbits on $D$, Theorem \ref{T:hd} is also a classification $G(\bR)$--conjugacy classes of horizontal $\tSL(2)$--orbits on period domains.
\end{remark}

\begin{remark} \label{R:polorb}
There is an interesting relationship between PMHS and the topological boundary $\partial(D) \subset \check D$ of $D$.  First, recall (Remark \ref{R:asymlim}) that the asymptotic limit
\[
  \Phi_\infty(N,F) \ = \ 
  \lim_{\mystack{\tIm\,z\to\infty}{\tRe\,z \ \mathrm{bdd}}} 
  \exp(z N) \cdot F \ \in \ \partial(D)
\]
of the nilpotent orbit lies in the boundary if $N\not=0$.  Note that $\partial(D)$ decomposes into a disjoint union of $G$ orbits, and the map $\Phi_\infty$ is $G$--equivariant.  Thus, from Theorem \ref{T:hd}, we obtain an induced map $\Phi_\infty$ from the set $\Diamond(D)$ of Hodge diamonds to the set $\cO(D)$ of $G$--orbits in $\overline D$.  This map is injective \cite{KPR}.  In particular, the Hodge diamonds index the ``polarizable" $G$--orbits in $\overline D$.  In \cite{GGR} the natural closure relations between these orbits are used to define the notion of an extremal degeneration of Hodge structures.  Theorem \ref{T:PDpr} (and more generally the results of \cite{KPR}) may be viewed as refining, or ``filling-in", results of \cite{GGR}.
\end{remark}

We finish this section by giving a number of examples illustrating Theorem \ref{T:hd}.  In each of the examples that follows we fix a period domain $D$ (by specifying the Hodge numbers and fixing a polarization $Q$), and apply Theorem \ref{T:hd} to list the Hodge diamonds.  The diamonds are represented by labeled configurations of points in the $pq$--plane: the node at $(p,q)$ is labeled with the (nonzero) value of $\Diamond(p,q)$.

\subsubsection*{Notation}
The following notation will be used to characterize the flags $F^\sb \in \check D$ realizing a given Hodge diamond in the examples below.  Observe that 
\[
  Q^*(\cdot,\cdot) \ := \ \bi^n\,Q(\cdot , \bar\cdot)
\]
defines a nondegenerate Hermitian form on $V_\bC$.  Given a subspace $E \subset V_\bC$, let 
\[
  E_0 \ := \ \{ v \in E \ | \ Q^*(v,E) = 0 \} \,.
\]
Notice that $Q^*$ induces a nondegenerate Hermitian form $Q^*_0$ on $E/E_0$.  We will write $Q^*_0 > 0$ when this form is positive definite.

\begin{example}[Hodge numbers $\bh = (g,g)$] \label{eg:gg}
There are $g+1$ Hodge diamonds, which we denote $\Diamond_r$, with $0 \le r \le g$, and picture as
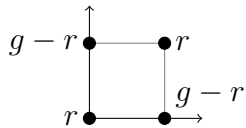
\begin{center}
\begin{tikzpicture}
\draw [<->] (0,1.5) -- (0,0) -- (1.5,0);
\draw [gray] (1,0) -- (1,1);
\draw [gray] (0,1) -- (1,1);
\draw [fill] (0,0) circle [radius=0.08];
\node [left] at (0,0) {$r$};
\draw [fill] (0,1) circle [radius=0.08];
\node [left] at (0,1) {$g-r$};
\draw [fill] (1,0) circle [radius=0.08];
\node [above right] at (1,0) {$g-r$};
\draw [fill] (1,1) circle [radius=0.08];
\node [right] at (1,1) {$r$};
\end{tikzpicture}
\end{center}
In this case $Q$ is skew-symmetric, and $Q^*(\cdot , \cdot) = \bi \,Q(\cdot,\bar\cdot)$ is a nondegenerate Hermitian form on $V_\bC \simeq \bC^{2g}$.  The flags $F^\sb = (F^1) \in \check D$ consist of a single subspace and the compact dual $\check D = \tGr^Q(g,V_\bC)$ is the Lagrangian grassmannian.  The subspaces $E \in \check D$ realizing the Hodge diamond $\Diamond_r$ form a $G$--orbit 
\[
  \cO_r \ := \ \{ E \in \tGr^Q(g,V_\bC) \ | \ E_0 = E \cap \overline E \,,
  \tdim\,E_0 = r \,,\ Q^*_0 > 0 \} \,.
\]
Note that $D = \cO_0$ and $\cO_s \subset \overline\cO_r$ if and only if $r \le s$.
\end{example}

\begin{remark}
Recall Figure \ref{fig:g2} and the associated limit mixed Hodge structures (which are PMHS).  Here we have $g=2$, and the Hodge diamond for LMHS$(t_1,0)$ and LMHS$(0,t_2)$, with $t_1t_2\not=0$, is $\Diamond_1$; likewise, the Hodge diamond for LMHS$(0,0)$ is $\Diamond_2$.
\end{remark}

\begin{example}[Hodge numbers $\bh = (1,a,1)$] \label{eg:1a1}
The Hodge diamonds are 
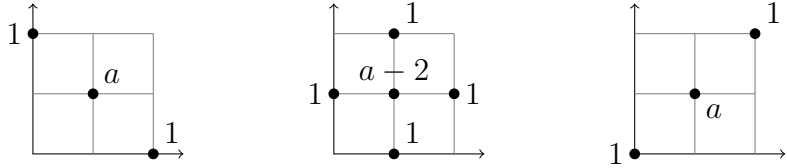
\begin{center}
\begin{tikzpicture}[scale=0.8]
\begin{scope}[xshift=-5cm]
  \draw [<->] (0,2.5) -- (0,0) -- (2.5,0);
  \draw [gray] (1,0) -- (1,2);
  \draw [gray] (2,0) -- (2,2);
  \draw [gray] (0,1) -- (2,1);
  \draw [gray] (0,2) -- (2,2);
  \draw [fill] (0,2) circle [radius=0.08];
  \node [left] at (0,2) {$1$};
  \draw [fill] (1,1) circle [radius=0.08];
  \node [above right] at (1,1) {$a$};
  \draw [fill] (2,0) circle [radius=0.08];
  \node [above right] at (2,0) {$1$};
\end{scope}
  \draw [<->] (0,2.5) -- (0,0) -- (2.5,0);
  \draw [gray] (1,0) -- (1,2);
  \draw [gray] (2,0) -- (2,2);
  \draw [gray] (0,1) -- (2,1);
  \draw [gray] (0,2) -- (2,2);
  \draw [fill] (0,1) circle [radius=0.08];
  \node [left] at (0,1) {$1$};
  \draw [fill] (1,2) circle [radius=0.08];
  \node [above right] at (1,2) {$1$};
  \draw [fill] (1,1) circle [radius=0.08];
  \node [above] at (1,1) {$a-2$};
  \draw [fill] (1,0) circle [radius=0.08];
  \node [above right] at (1,0) {$1$};
  \draw [fill] (2,1) circle [radius=0.08];
  \node [right] at (2,1) {$1$};
\begin{scope}[xshift=5cm]
  \draw [<->] (0,2.5) -- (0,0) -- (2.5,0);
  \draw [gray] (1,0) -- (1,2);
  \draw [gray] (2,0) -- (2,2);
  \draw [gray] (0,1) -- (2,1);
  \draw [gray] (0,2) -- (2,2);
  \draw [fill] (2,2) circle [radius=0.08];
  \node [above right] at (2,2) {$1$};
  \draw [fill] (1,1) circle [radius=0.08];
  \node [below right] at (1,1) {$a$};
  \draw [fill] (0,0) circle [radius=0.08];
  \node [left] at (0,0) {$1$};
\end{scope}
\end{tikzpicture}
\end{center}
Here the second diamond arises only if $a \ge 2$, and the third only if $a \ge 1$.
\end{example}

\begin{example}[Hodge numbers $\bh=(b,a,b)$] \label{eg:bab}
The Hodge diamonds $\Diamond_{r,s}$ are indexed by $0\le r,s$ satisfying $r+s \le b$ and $r+2s \le a$.
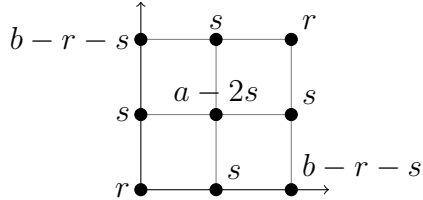
\begin{center}
\begin{tikzpicture}
  \draw [<->] (0,2.5) -- (0,0) -- (2.5,0);
  \draw [gray] (1,0) -- (1,2);
  \draw [gray] (2,0) -- (2,2);
  \draw [gray] (0,1) -- (2,1);
  \draw [gray] (0,2) -- (2,2);
  \draw [fill] (0,2) circle [radius=0.08];
  \node [left] at (0,2) {$b-r-s$};
  \draw [fill] (0,1) circle [radius=0.08];
  \node [left] at (0,1) {$s$};
  \draw [fill] (0,0) circle [radius=0.08];
  \node [left] at (0,0) {$r$};
  \draw [fill] (1,2) circle [radius=0.08];
  \node [above] at (1,2) {$s$};
  \draw [fill] (1,1) circle [radius=0.08];
  \node [above] at (1,1) {$a-2s$};
  \draw [fill] (1,0) circle [radius=0.08];
  \node [above right] at (1,0) {$s$};
  \draw [fill] (2,2) circle [radius=0.08];
  \node [above right] at (2,2) {$r$};
  \draw [fill] (2,1) circle [radius=0.08];
  \node [above right] at (2,1) {$s$};
  \draw [fill] (2,0) circle [radius=0.08];
  \node [above right] at (2,0) {$b-r-s$};
\end{tikzpicture}
\end{center}
In this case $Q$ is symmetric, and $Q^*(\cdot , \cdot) = -Q(\cdot,\bar\cdot)$ defines a nondegenerate Hermitian form on $V_\bC \simeq \bC^{a+2b}$.  The flags $(F^2 \subset F^1) \in \check D = \tFlag^Q(b,a+b,V_\bC)$ satisfy $F^1 = (F^2)^\perp$; that is, the flag $F^\sb$ is completely determined by the first subspace $F^2$, so that $\check D \simeq \tGr^Q(b,\bC^{a+2b})$.  The flags realizing the Hodge diamond $\Diamond_{r,s}$ form a $G$--orbit $\cO_{r,s} \in \overline D$ that is characterized by 
\[
  \cO_{r,s} \ = \ \left\{
    E \in \tGr^Q(b,\bC^{a+2b}) \ \left| \ 
    \tdim\,E\cap\overline E = r \,,\ \tdim\,E_0 = r+s \,,\ Q^*_0 > 0
  \right.\right\}
\]
Note that $D = \cO_{0,0}$ and $\cO_{t,u} \subset \overline\cO_{r,s}$ if and only if $r \le t$ and $r+s \le t+u$.
\end{example}

\subsection{Degeneracy relations between $\bR$--split PMHS} \label{S:classpr}

\subsubsection{Polarized relations on Hodge diamonds}

Schmid's Nilpotent Orbit Theorem \ref{T:norb} provides the link between the geometry and the Hodge theory.  Specifically, the lift $\tilde\Phi : \sH \times \sH \to D$ of any period map $\Phi : \Delta^* \times \Delta^* \to \Gamma\backslash D$ with unipotent monodromies is a approximated by a two-variable nilpotent orbit 
\begin{equation}\label{E:pr}
  \theta(z_1,z_2) \ = \ \exp(z_1 N_1 + z_2 N_2) \cdot F\,.
\end{equation}
We may assume, without loss of generality, that the associated PMHS $(F,\s)$, with $\s = \{ x_1 N_1 + x_2 N_2 \ | \ x_j > 0 \}$ the underlying nilpotent orbit, is $\bR$--split (Theorem \ref{T:cks}).  Note that $z_1 \mapsto \theta(z_1 , \bi) = \exp(z_1 N_1) e^{\bi N_2} \cdot F$ is a one-variable nilpotent orbit with corresponding PMHS $(e^{\bi N_2} \cdot F , N_1)$.  Fixing $z_2 = \bi \in \sH$ and letting $\tIm\,z_1 \to \infty$ corresponds to fixing $t_2 \in \Delta^*$ and letting $t_1 \to 0$.  So it is natural to regard the PMHS $(e^{\bi N_2} \cdot F , N_1)$ as ``less degenerate'' than $(F,N)$.  (Cf. Figure \ref{fig:g2} and the related discussion.)  This motivates the following definition: given \emph{any} $\bR$--split two-variable nilpotent orbit \eqref{E:pr}, let $(F_1,N_1)$ be the $\bR$--split PMHS associated to $(e^{\bi N_2} \cdot F , N_1)$ as in \S\ref{S:ds}.  Then we say the corresponding Hodge diamonds satisfy the \emph{polarized relation} 
\[
  \Diamond(F_1 , N_1) \ \preceq \ \Diamond(F,N) \,.
\]
The polarized relations are classified in Theorem \ref{T:PDpr}.  The classification requires the notion of a ``primitive Hodge diamond.''

\begin{remark}
Recall the map $\Phi_\infty : \Diamond(D) \to \cO(D)$ of Remark \ref{R:polorb}.  Observe that the $G$--orbit $\cO(F,N) = \Phi_\infty(\Diamond(F,N)) \subset \overline D$ is contained in the closure of $\cO(F_1,N_1) = \Phi_\infty(\Diamond(F_1,N_1))$.  In particular, if we define a partial order on the set $\cO(D)$ of $G$--orbits in $\overline D$ by $\cO_1 \le \cO$ if $\cO \subset \overline\cO_1$, then the map $\Phi_\infty$ preserves the two relations.  (However, beware that the polarized relation on Hodge diamonds is not, in general, a partial order: transitivity may fail.  See Example \ref{eg:1aa1}.)
\end{remark}

\subsubsection{Primitive subspaces} \label{S:prim}

Fix a $\bR$--split PMHS $(F,N)$, and let $V_\bC = \op I^{p,q}$ be the Deligne splitting (\S\ref{S:ds}).  Set
\begin{subequations}\label{SE:Pell}
\begin{equation} \label{E:Ppq}
  P(N)^{p,q} \ := \ \tker\{ N^{\ell+1} : I^{p,q} \to I^{-p-1,-q-1} \} \,,
\end{equation}
and define the \emph{weight $n+\ell$ $N$--primitive subspace}
\begin{equation}
  P(N)_{n+\ell,\bC} \ := \ 
  \displaystyle\bigoplus_{p+q=n+\ell} P(N)^{p,q} \,.
\end{equation}
\end{subequations}
Since $(F,N)$ is $\bR$--split we see that the $P(N)_{n+\ell}$ is defined over $\bR$.  Moreover, from the second equation of \eqref{E:dsFW} and Remark \ref{R:sl2} it may be deduced that
\begin{equation}\label{E:Ndecomp}
  V_\bR \ = \ 
  \bigoplus_{\mystack{0\le \ell}{0 \le a \le \ell}} N^a P(N)_{n+\ell} \,.
\end{equation}
In particular, the decomposition \eqref{SE:Pell} determines the Deligne bigrading $V_\bC = \op\,I^{p,q}$ of $(F,W(N))$.  Moreover, \eqref{E:Ppq} is a weight $n+\ell$ Hodge decomposition of $P(N)_{\ell,\bR}$ polarized by 
\[
  Q^N_\ell(\cdot,\cdot) \ := \ Q(\cdot , N^\ell\cdot)\,.
\]
The \emph{$N$--primitive Hodge--Deligne numbers} are the 
\[
  j^{p,q} \ := \ \tdim_\bC\,P(N)^{p,q}\,.
\]
The \emph{weight $n+\ell$ primitive part} of $\Diamond(F,N)$ is the function
\[
  \Diamond^\tprim_{n+\ell}(F,N) : \bZ \times \bZ \ \to \ \bZ_{\ge0}
  \quad\hbox{sending}\quad
  (p,q) \ \mapsto \ j^{p,q} \,.
\]
Likewise, the \emph{primitive part} of $\Diamond(F,N)$ is the sum
\[
  \Diamond^\tprim(F,N) \ = \ \sum_{\ell=0}^n
  \Diamond^\tprim_{n+\ell}(F,N)
\] 
of the weight $k$ primitive Hodge sub-diamonds.  Note that $\Diamond^\tprim_{n+\ell}(F,N)$ not a Hodge diamond: \eqref{E:hd3} and \eqref{E:hd4} will fail whenever $N\not=0$.  We will call any such $\Diamond^\tprim(F,N)$ a \emph{primitive sub-diamond for the period domain $D$}.  From \eqref{E:Ndecomp} we see that 
\begin{equation}\label{E:pHD}
  \hbox{\emph{$\Diamond^\tprim(F,N)$ determines 
  $\Diamond(F,N)$ (and visa versa).}}
\end{equation}
To be more precise, given $f : \bZ \times \bZ \to \bZ_{\ge0}$ define $f[\ell]: \bZ \times \bZ \to \bZ_{\ge0}$ by $(p,q) \mapsto f(p+\ell,q+\ell)$.  Then \eqref{E:Ndecomp} implies
\begin{equation}\nonumber
  \Diamond(F,N) \ = \ \sum_{\mystack{0 \le \ell}{0 \le a\le \ell}} 
  \Diamond^\tprim_{n+\ell}(F,N)[a] \,.
\end{equation}
From Theorem \ref{T:hd} we then obtain

\begin{corollary}\label{C:hd}
The $G$--conjugacy class of an $\bR$--split PMHS $(F,N)$ on $D$ is determined by the primitive sub-diamond $\Diamond^\tprim(F,N)$.
\end{corollary}

\begin{example}[Hodge numbers $\bh = (g,g)$] \label{eg:ggprim}
The primitive Hodge diamond for the $\bR$--split PMHS $(F_r,N_r)$ with Hodge diamond $\Diamond_r$ of Example \ref{eg:gg} is 
\begin{center}
\begin{tikzpicture}
\draw [<->] (0,1.5) -- (0,0) -- (1.5,0);
\draw [gray] (1,0) -- (1,1);
\draw [gray] (0,1) -- (1,1);
\draw [fill] (0,1) circle [radius=0.08];
\node [left] at (0,1) {$g-r$};
\draw [fill] (1,0) circle [radius=0.08];
\node [above right] at (1,0) {$g-r$};
\draw [fill] (1,1) circle [radius=0.08];
\node [right] at (1,1) {$r$};
\end{tikzpicture}
\end{center}
\end{example}

\begin{example}[Hodge numbers $\bh = (b,a,b)$] \label{eg:babprim}
The primitive Hodge diamond for the $\bR$--split PMHS $(F_{r,s},N_{r,s})$ with Hodge diamond $\Diamond_{r,s}$ of Example \ref{eg:bab} is 
\begin{center}
\begin{tikzpicture}
  \draw [<->] (0,2.5) -- (0,0) -- (2.5,0);
  \draw [gray] (1,0) -- (1,2);
  \draw [gray] (2,0) -- (2,2);
  \draw [gray] (0,1) -- (2,1);
  \draw [gray] (0,2) -- (2,2);
  \draw [fill] (0,2) circle [radius=0.08];
  \node [left] at (0,2) {$b-r-s$};
  \draw [fill] (1,2) circle [radius=0.08];
  \node [above] at (1,2) {$s$};
  \draw [fill] (1,1) circle [radius=0.08];
  \node [below left] at (1,1) {$a-r-2s$};
  \draw [fill] (2,2) circle [radius=0.08];
  \node [above right] at (2,2) {$r$};
  \draw [fill] (2,1) circle [radius=0.08];
  \node [above right] at (2,1) {$s$};
  \draw [fill] (2,0) circle [radius=0.08];
  \node [above right] at (2,0) {$b-r-s$};
\end{tikzpicture}
\end{center}
\end{example}

\begin{theorem}[{\cite{KPR}}] \label{T:PDpr}
Let $D$ be a period domain parameterizing weight $n$, $Q$--polarized Hodge structures on $V_\bR$.  Let $[F_1,N_1], [F_2,N_2] \in \Psi_D$.  Then $[F_1,N_1] \preceq [F_2,N_2]$ if and only if $\Diamond(F_2,N_2)$ can be expressed as a sum 
\[
  \Diamond(F_2,N_2) \ = \ \sum_{\mystack{0\le k}{0 \le a \le \ell}}  
  \Diamond(F'_\ell,N'_\ell)[a]
\]
for Hodge diamonds $\Diamond(F'_\ell,N'_\ell)$ on the period domains $D_\ell$ parameterizing weight $n+\ell$, $Q^{N_1}_\ell$--polarized Hodge structures on $P(N_1)_{n+\ell}$ with Hodge numbers $\{j_1^{p,q} \ | \ p+q = n+\ell \}$.
\end{theorem}

\begin{proof}
The theorem is a consequence of the Cattani--Kaplan--Schmid Several Variable $\tSL(2)$--Orbit Theorem \cite{\CKSdeg}.  See \cite{KPR} for details.
\end{proof}

We finish this section by giving a number of examples illustrating Theorem \ref{T:PDpr}.

\begin{example}[Hodge numbers $\bh = (g,g)$] \label{eg:gg'}
The polarized relations among the Hodge diamonds in Example \ref{eg:gg} are $\Diamond_s \prec \Diamond_r$ if and only if $s > r$.  In particular, in this case $\prec$ is a linear order.  Moreover, the map $\Phi_\infty: \Diamond(D) \to \cO(D)$ of Remark \ref{R:polorb} sending $\Diamond_r \mapsto \cO_r$ is a bijection preserving the orders.
\end{example}

\begin{remark}\label{R:gg}
Notice that each of the polarized relations in Example \ref{eg:gg'} can be realized geometrically.  For example, if $g=2$, then the polarized relations $\Diamond_0 \prec \Diamond_1 \prec \Diamond_2$ can be realized by degenerating a curve of genus two so that it acquires one, and then two, nodes.
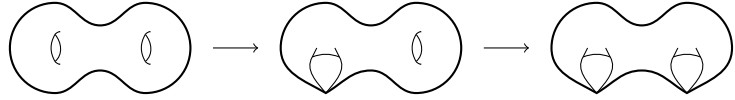
\begin{center}
\begin{tikzpicture}[scale=0.6]
\draw[thick] (0,0) to [out=90,in=180] (1,1)
  to [out=0,in=180] (2,0.5)
  to [out=0,in=180] (3,1)
  to [out=0,in=90] (4,0)
  to [out=270,in=0] (3,-1)
  to [out=180,in=0] (2,-0.5)
  to [out=180,in=0] (1,-1)
  to [out=180,in=270] (0,0);
\draw (1,0.3) to [out=240,in=120] (1,-0.3);
\draw (1,0.3) to [out=60,in=50] (1.1,0.35);
\draw (1,-0.3) to [out=300,in=310] (1.1,-0.35);
\draw (1,0.3) to [out=310,in=40] (1,-0.3);
\draw (3,0.3) to [out=240,in=120] (3,-0.3);
\draw (3,0.3) to [out=60,in=50] (3.1,0.35);
\draw (3,-0.3) to [out=300,in=310] (3.1,-0.35);
\draw (3,0.3) to [out=310,in=40] (3,-0.3);
\draw[->] (4.5,0) -- (5.5,0);
\draw[thick] (6,0) to [out=90,in=180] (7,1)
  to [out=0,in=180] (8,0.5)
  to [out=0,in=180] (9,1)
  to [out=0,in=90] (10,0)
  to [out=270,in=0] (9,-1)
  to [out=180,in=0] (8,-0.5)
  to [out=180,in=30] (7,-1)
  to [out=150,in=270] (6,0);
\draw (7,-1) to [out=130,in=240] (6.7,-0.2)
  to [out=60,in=240] (6.8,0);
\draw (7,-1) to [out=50,in=300] (7.3,-0.2)
  to [out=120,in=300] (7.2,0);
\draw (6.7,-0.2) to [out=20,in=160] (7.3,-0.2);
\draw (9,0.3) to [out=240,in=120] (9,-0.3);
\draw (9,0.3) to [out=60,in=50] (9.1,0.35);
\draw (9,-0.3) to [out=300,in=310] (9.1,-0.35);
\draw (9,0.3) to [out=310,in=40] (9,-0.3);
\draw[->] (10.5,0) -- (11.5,0);
\draw[thick] (12,0) to [out=90,in=180] (13,1)
  to [out=0,in=180] (14,0.5)
  to [out=0,in=180] (15,1)
  to [out=0,in=90] (16,0)
  to [out=270,in=30] (15,-1)
  to [out=150,in=0] (14,-0.5)
  to [out=180,in=30] (13,-1)
  to [out=150,in=270] (12,0);
\draw (13,-1) to [out=130,in=240] (12.7,-0.2)
  to [out=60,in=240] (12.8,0);
\draw (13,-1) to [out=50,in=300] (13.3,-0.2)
  to [out=120,in=300] (13.2,0);
\draw (12.7,-0.2) to [out=20,in=160] (13.3,-0.2);
\draw (15,-1) to [out=130,in=240] (14.7,-0.2)
  to [out=60,in=240] (14.8,0);
\draw (15,-1) to [out=50,in=300] (15.3,-0.2)
  to [out=120,in=300] (15.2,0);
\draw (14.7,-0.2) to [out=20,in=160] (15.3,-0.2);
\end{tikzpicture}
\end{center}
\end{remark}

\begin{example}[Hodge numbers $\bh=(b,a,b)$]\label{eg:bab'}
The polarized relations amongst the Hodge diamonds in Example \ref{eg:bab} are $\Diamond_{r,s} \preceq \Diamond_{t,u}$ if and only if $r \le t$ and $r+s \le t+u$.  As in Example \ref{eg:gg'} the map $\Phi_\infty: \Diamond(D) \to \cO(D)$ of Remark \ref{R:polorb} sending $\Diamond_{r,s} \mapsto \cO_{r,s}$ is a bijection preserving the relations.  In particular, if $b=1$, then $\prec$ is a linear order.  However, if $b > 1$, then $\prec$ is a partial order, but not a linear order.  For example, if $b = 2$ and $a \ge 4$, then the partial order is represented by the diagram
\begin{equation} \label{E:2a2} 
\begin{tikzpicture}[baseline=(current  bounding  box.center)]
\path (0,0) node(00) {$\Diamond_{0,0}$}
	(2,0) node(01) {$\Diamond_{0,1}$}
	(4,0.5) node(02) {$\Diamond_{0,2}$}
	(4,-0.5) node(10) {$\Diamond_{1,0}$}
	(6,0) node(11) {$\Diamond_{1,1}$}
	(8,0) node(20) {$\Diamond_{2,0}$};
\draw[->] (00) -- (01);
\draw[->] (01) -- (02);
\draw[->] (01) -- (10);
\draw[->] (02) -- (11);
\draw[->] (10) -- (11);
\draw[->] (11) -- (20);
\end{tikzpicture}
\end{equation}
with, for example, $\Diamond_{0,0} \to \Diamond_{0,1}$ indicating $\Diamond_{0,0} \prec \Diamond_{0,1}$.  Likewise for $b = 3$ and $a\ge6$ we have 
\begin{center}
\begin{tikzpicture}
\path (0,0) node(00) {$\Diamond_{0,0}$}
	(2,0) node(01) {$\Diamond_{0,1}$}
	(4,1) node(10) {$\Diamond_{1,0}$}
	(4,-1) node(02) {$\Diamond_{0,2}$}
	(6,1) node(11) {$\Diamond_{1,1}$}
	(6,-1) node(03) {$\Diamond_{0,3}$}
	(8,1) node(20) {$\Diamond_{2,0}$}
	(8,-1) node(12) {$\Diamond_{1,2}$}
	(10,0) node(21) {$\Diamond_{2,1}$}
	(12,0) node(30) {$\Diamond_{3,0}$};
\draw[->] (00) -- (01);
\draw[->] (01) -- (10);
\draw[->] (01) -- (02);
\draw[->] (10) -- (11);
\draw[->] (02) -- (11);
\draw[->] (02) -- (03);
\draw[->] (11) -- (20);
\draw[->] (11) -- (12);
\draw[->] (03) -- (12);
\draw[->] (20) -- (21);
\draw[->] (12) -- (21);
\draw[->] (21) -- (30);
\end{tikzpicture}
\end{center}
\end{example}

\begin{remark}
The period domains for polarized Hodge structures with Hodge numbers $\bh = (g,g)$ or $\bh = (1,a,1)$ are Hermitian symmetric period domains.  More generally, if $D$ is any Hermitian symmetric Mumford--Tate domain, then $\prec$ is a linear order \cite{KPR}.  
\end{remark}

\begin{remark}[Geometric realization of the polarized relations for $\bh=(2,27,2)$]\label{R:2a2}
In analogy with Remark \ref{R:gg}, each of the polarized relations of \eqref{E:2a2} by degenerations of Horikawa surfaces; see the forthcoming \cite{GGLR} for details.
\end{remark}

\begin{problem}
In analogy with Remarks \ref{R:gg} and \ref{R:2a2}, give geometric realizations of the polarized relations for other values of $\bh$.
\end{problem}

\begin{example}[Hodge numbers $\bh = (1,2,2,1)$] \label{eg:1aa1}
From Theorem \ref{T:hd} we see that the Hodge diamonds include the following three
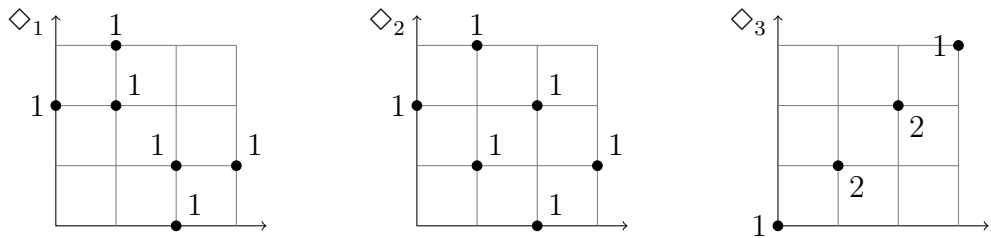
\begin{center}
\begin{tikzpicture}[scale=0.8]
\begin{scope}[xshift=-6cm]
  \draw [<->] (0,3.5) -- (0,0) -- (3.5,0);
  \draw [gray] (1,0) -- (1,3);
  \draw [gray] (2,0) -- (2,3);
  \draw [gray] (3,0) -- (3,3);
  \draw [gray] (0,1) -- (3,1);
  \draw [gray] (0,2) -- (3,2);
  \draw [gray] (0,3) -- (3,3);
  \draw [fill] (0,2) circle [radius=0.08];
  \node [left] at (0,2) {$1$};
  \draw [fill] (1,3) circle [radius=0.08];
  \node [above] at (1,3) {$1$};
  \draw [fill] (1,2) circle [radius=0.08];
  \node [above right] at (1,2) {$1$};
  \draw [fill] (2,0) circle [radius=0.08];
  \node [above right] at (2,0) {$1$};
  \draw [fill] (2,1) circle [radius=0.08];
  \node [above left] at (2,1) {$1$};
  \draw [fill] (3,1) circle [radius=0.08];
  \node [above right] at (3,1) {$1$};
  \node [above left] at (0,3) {$\Diamond_1$};
\end{scope}
  \draw [<->] (0,3.5) -- (0,0) -- (3.5,0);
  \draw [gray] (1,0) -- (1,3);
  \draw [gray] (2,0) -- (2,3);
  \draw [gray] (3,0) -- (3,3);
  \draw [gray] (0,1) -- (3,1);
  \draw [gray] (0,2) -- (3,2);
  \draw [gray] (0,3) -- (3,3);
  \draw [fill] (0,2) circle [radius=0.08];
  \node [left] at (0,2) {$1$};
  \draw [fill] (1,3) circle [radius=0.08];
  \node [above] at (1,3) {$1$};
  \draw [fill] (1,1) circle [radius=0.08];
  \node [above right] at (1,1) {$1$};
  \draw [fill] (2,0) circle [radius=0.08];
  \node [above right] at (2,0) {$1$};
  \draw [fill] (2,2) circle [radius=0.08];
  \node [above right] at (2,2) {$1$};
  \draw [fill] (3,1) circle [radius=0.08];
  \node [above right] at (3,1) {$1$};
  \node [above left] at (0,3) {$\Diamond_2$};
\begin{scope}[xshift=6cm]
  \draw [<->] (0,3.5) -- (0,0) -- (3.5,0);
  \draw [gray] (1,0) -- (1,3);
  \draw [gray] (2,0) -- (2,3);
  \draw [gray] (3,0) -- (3,3);
  \draw [gray] (0,1) -- (3,1);
  \draw [gray] (0,2) -- (3,2);
  \draw [gray] (0,3) -- (3,3);
  \draw [fill] (2,2) circle [radius=0.08];
  \node [below right] at (2,2) {$2$};
  \draw [fill] (1,1) circle [radius=0.08];
  \node [below right] at (1,1) {$2$};
  \draw [fill] (0,0) circle [radius=0.08];
  \node [left] at (0,0) {$1$};
  \draw [fill] (3,3) circle [radius=0.08];
  \node [left] at (3,3) {$1$};
  \node [above left] at (0,3) {$\Diamond_3$};
\end{scope}
\end{tikzpicture}
\end{center}
From Theorem \ref{T:PDpr} we see that $\Diamond_1 \prec \Diamond_2 \prec \Diamond_3$, but $\Diamond_1 \not\prec \Diamond_3$.  That is, \emph{the polarized relation is not transitive}, and therefore can not define a partial order.  There is a similar failure of transitivity for $\bh = (1,a,a,1)$, $a \ge 2$, cf.~\cite{KPR}.
\end{example}

\begin{question}
What are the geometric implications of the failure of transitivity in the Calabi--Yau 3-fold case (Example \ref{eg:1aa1})?
\end{question}

\begin{problem}
Identify the Hodge diamonds for $\bh = (1,a,a,1)$, and determine their polarized relations \cite{KPR}.
\end{problem}

\begin{remark}
Work of Bloch, Kerr and Vanhove \cite{BKV} yields geometric realizations of the polarized relations on Calabi--Yau 3-folds, see \cite{KPR}.
\end{remark}

\begin{remark}
In the case that $D$ is a period domain parameterizing polarized Hodge structures with Hodge numbers $\bh = (1,\ldots,1)$ or $\bh = (1,\ldots,1,2,1,\ldots,1)$, then $\prec$ is a partial order.  In both these cases, the isotropy group (the subgroup of $G$ stabilizing $\varphi \in D$) is a compact torus (equivalently the complex parabolic subgroup of $G_\bC$ stabilizing the Hodge filtration is a Borel).  More generally, if $D$ is a Mumford--Tate domain with isotropy a compact torus, then $\prec$ is a partial order \cite{KPR}.
\end{remark}

\input{bootcamp.bbl}

\end{document}

%% file: bootcamp.bbl
\def\cprime{$'$} \def\Dbar{\leavevmode\lower.6ex\hbox to 0pt{\hskip-.23ex
  \accent"16\hss}D}

%% file: bootcamp.bbl
\begin{thebibliography}{10}

\bibitem{BKV}
Spencer Bloch, Matt Kerr, and Pierre Vanhove.
\newblock Local mirror symmetry and the sunset feynman integral.
\newblock arxiv:601.08181, 2016.

\bibitem{BPR}
P.~Brosnan, G.~Pearlstein, and C.~Robles.
\newblock Nilpotent cones and their representation theory.
\newblock In preparation; an expository chapter in a collection honoring Steve
  Zucker's work, 2015.

\bibitem{MR2012297}
James Carlson, Stefan M{\"u}ller-Stach, and Chris Peters.
\newblock {\em Period mappings and period domains}, volume~85 of {\em Cambridge
  Studies in Advanced Mathematics}.
\newblock Cambridge University Press, Cambridge, 2003.

\bibitem{MR3288678}
Eduardo Cattani, Fouad El~Zein, Phillip~A. Griffiths, and L{\^e}~D{\~u}ng
  Tr{{\'a}}ng, editors.
\newblock {\em Hodge theory}, volume~49 of {\em Mathematical Notes}.
\newblock Princeton University Press, Princeton, NJ, 2014.

\bibitem{MR664326}
Eduardo Cattani and Aroldo Kaplan.
\newblock Polarized mixed {H}odge structures and the local monodromy of a
  variation of {H}odge structure.
\newblock {\em Invent. Math.}, 67(1):101--115, 1982.

\bibitem{MR840721}
Eduardo Cattani, Aroldo Kaplan, and Wilfried Schmid.
\newblock Degeneration of {H}odge structures.
\newblock {\em Ann. of Math. (2)}, 123(3):457--535, 1986.

\bibitem{MR0444662}
C.~H. Clemens.
\newblock Degeneration of {K}{\"a}hler manifolds.
\newblock {\em Duke Math. J.}, 44(2):215--290, 1977.

\bibitem{MR2931865}
Olivier Debarre.
\newblock Periods and moduli.
\newblock In {\em Current developments in algebraic geometry}, volume~59 of
  {\em Math. Sci. Res. Inst. Publ.}, pages 65--84. Cambridge Univ. Press,
  Cambridge, 2012.

\bibitem{MR0498552}
Pierre Deligne.
\newblock Th{\'e}orie de {H}odge. {III}.
\newblock {\em Inst. Hautes {\'E}tudes Sci. Publ. Math.}, (44):5--77, 1974.

\bibitem{MR713069}
Alan~H. Durfee.
\newblock A naive guide to mixed {H}odge theory.
\newblock In {\em Singularities, {P}art 1 ({A}rcata, {C}alif., 1981)},
  volume~40 of {\em Proc. Sympos. Pure Math.}, pages 313--320. Amer. Math.
  Soc., Providence, RI, 1983.

\bibitem{GGLR}
M.~Green, P.~Griffiths, R.~Laza, and C.~Robles.
\newblock Hodge theory and moduli of {H}-surfaces.
\newblock In preparation.

\bibitem{MR2918237}
Mark Green, Phillip Griffiths, and Matt Kerr.
\newblock {\em Mumford-{T}ate groups and domains: their geometry and
  arithmetic}, volume 183 of {\em Annals of Mathematics Studies}.
\newblock Princeton University Press, Princeton, NJ, 2012.

\bibitem{GGR}
Mark Green, Phillip Griffiths, and Colleen Robles.
\newblock Extremal degenerations of polarized {H}odge structures.
\newblock Preprint, arXiv:1403.0646, 2014.

\bibitem{MR0259958}
Phillip Griffiths and Wilfried Schmid.
\newblock Locally homogeneous complex manifolds.
\newblock {\em Acta Math.}, 123:253--302, 1969.

\bibitem{MR0229641}
Phillip~A. Griffiths.
\newblock Periods of integrals on algebraic manifolds. {I}. {C}onstruction and
  properties of the modular varieties.
\newblock {\em Amer. J. Math.}, 90:568--626, 1968.

\bibitem{MR0233825}
Phillip~A. Griffiths.
\newblock Periods of integrals on algebraic manifolds. {II}. {L}ocal study of
  the period mapping.
\newblock {\em Amer. J. Math.}, 90:805--865, 1968.

\bibitem{MR0199184}
Heisuke Hironaka.
\newblock Resolution of singularities of an algebraic variety over a field of
  characteristic zero. {I}, {II}.
\newblock {\em Ann. of Math. (2) {\bf 79} (1964), 109--203; ibid. (2)},
  79:205--326, 1964.

\bibitem{MR0291177}
Nicholas~M. Katz.
\newblock Nilpotent connections and the monodromy theorem: {A}pplications of a
  result of {T}urrittin.
\newblock {\em Inst. Hautes \'Etudes Sci. Publ. Math.}, (39):175--232, 1970.

\bibitem{KPR}
Matt Kerr and Colleen Robles.
\newblock Partial orders and polarized relations on limit mixed hodge
  structures.
\newblock In prepration, 2015.

\bibitem{MR0344248}
Alan Landman.
\newblock On the {P}icard-{L}efschetz transformation for algebraic manifolds
  acquiring general singularities.
\newblock {\em Trans. Amer. Math. Soc.}, 181:89--126, 1973.

\bibitem{LazaPers}
Radu Laza.
\newblock Perspectives on the construction and compactification of moduli
  spaces.
\newblock arXiv:1403.2105, March 2014.

\bibitem{MR2393625}
Chris A.~M. Peters and Joseph H.~M. Steenbrink.
\newblock {\em Mixed {H}odge structures}, volume~52 of {\em Ergebnisse der
  Mathematik und ihrer Grenzgebiete. 3. Folge. A Series of Modern Surveys in
  Mathematics [Results in Mathematics and Related Areas. 3rd Series. A Series
  of Modern Surveys in Mathematics]}.
\newblock Springer-Verlag, Berlin, 2008.

\bibitem{SL2}
Colleen Robles.
\newblock Classification of horizontal $\mathrm{SL}_2$'s.
\newblock {\em Compositio Math.}, 2015.
\newblock To appear, arXiv:1405.3163.

\bibitem{MR0382272}
Wilfried Schmid.
\newblock Variation of {H}odge structure: the singularities of the period
  mapping.
\newblock {\em Invent. Math.}, 22:211--319, 1973.

\bibitem{MR0429885}
Joseph Steenbrink.
\newblock Limits of {H}odge structures.
\newblock {\em Invent. Math.}, 31(3):229--257, 1975/76.

\end{thebibliography}
